\documentclass[oneside, 10pt]{amsart}

\usepackage{amsmath, amsfonts, amssymb, amsthm, graphics}

\newtheorem{theorem}{Theorem}[section]
\newtheorem{lemma}[theorem]{Lemma}
\newtheorem{corollary}[theorem]{Corollary}
\newtheorem{proposition}[theorem]{Proposition}

\theoremstyle{definition}

\numberwithin{equation}{section}

\begin{document}

\title[Representations for the Bloch type semi-norm]{Representations for the Bloch type semi-norm of Fr\'{e}chet differentiable mappings}

\author{Marijan Markovi\'{c}}

\begin{abstract}
In this paper we give some results  concerning Fr\'{e}chet differentiable  mappings between domains  in normed spaces with controlled growth. The results
are  mainly  motivated by Pavlovi\'{c}'s equality for the Bloch semi-norm of continuously differentiable mappings in the Bloch class  on the unit ball of
the Euclidean space as well as the very recent  Joci\'{c}'s   generalization of this result.
\end{abstract}

\subjclass[2010]{Primary 32A18, 30D45; Secondary 32A37}

\keywords{Bloch functions, normal functions, operator monotone functions, normed spaces, Banach spaces, Fr\'{e}chet differentiable  mappings}

\address{
Faculty of  Sciences and Mathematics\endgraf
University of Montenegro\endgraf
D\v{z}ord\v{z}a Va\v{s}ingtona bb\endgraf
81000 Podgorica\endgraf
Montenegro\endgraf}
\email{marijanmmarkovic@gmail.com}

\maketitle

\section{Introduction}
The starting  point of this paper is  the Pavlovi\'{c} expression for the Bloch     semi-norm  of a continuously differentiable function on the unit ball
${B}^m\subseteq \mathbf{R}^m$, obtained in 2008, and the  Joci\'{c}'s result concerning   Fr\'{e}chet differentiable mappings that satisfy the Bloch type
growth  condition obtained in 2019.

The following is a content of \cite[Theorem 2]{PAVLOVIC.PEMS}.

\begin{proposition}[see \cite{PAVLOVIC.PEMS}]\label{PR.PAVLOVIC}
For a complex-valued continuously differentiable function $f$ on the unit  ball $B^m$,  $m\ge 2$,  we have
\begin{equation*}
\|f\|_{\mathfrak {B}}  =  \sup_{(x,y)\in \Delta^c _ {B^m }}  (1-|x|^2)^{\frac12}(1-|y|^2)^{\frac12} \frac {|f(x) - f(y)|}{|x-y|},
\end{equation*}
where
\begin{equation*}
\|f\|_{\mathfrak {B} } = \sup_{x\in B^m } (1-|x|^2)  {\|d_f(x)\|}
\end{equation*}
is the  Bloch semi-norm of   the function   $f$.
\end{proposition}

In the above proposition  $d_f(x)$ is the differential of $f$      at $x\in  B^m$, and  $\|d_f(x)\|$  is the operator norm of  the linear mapping $d_f(x):
\mathbf{R}^m\rightarrow\mathbf{R}^2$.  For  a set $S$ we have denoted
\begin{equation*}
\Delta_S =\{(x,x):x\in S\}\subseteq   S\times S\quad\text{and}\quad  \Delta_S^c=S\times S\backslash  \Delta_ S.
\end{equation*}

In \cite[Remark 4]{PAVLOVIC.PEMS} the author conjectures  that his theorem remains valid  if one assumes  that $f$ is a continuous differentiable mapping
on the unit  ball of  a  Hilbert  space   with values in a Banach space. This is  confirmed by Joci\'{c} \cite[Corollary 3.4]{JOCIC.JFA}
even for Fr\'{e}chet differentiable mappings  on the unit ball of a Banach space that have  image  in   an  another   Banach space.

If $X$ is a normed space with a norm $\|\cdot\|_X$, we  denote  by ${B}_X = \{x\in X: \|x\|_X<1\}$  the unit ball in $X$, and $B_X(x,r) = x+r B_X$ is the
open ball with centre  $x$ and   radius $r$. The corresponding  metric on $X$  is  $d_X(x,y) = \|x-y\|_X$. For the sake of  simplicity we  will sometimes
drop  the indexes.

Assume that   $X$ and $Y$ are  normed spaces. Recall  that a mapping $f:\Omega\rightarrow Y$,    where  $\Omega\subseteq X$  is  a domain, is Fr\'{e}chet
differentiable at $x\in\Omega$  if there exists  a  continuous linear mapping $d_f(x):X\rightarrow Y$   (called  Fr\'{e}chet  differential), such that in
an  open ball  $B_X(x,r)$ we have
\begin{equation*}
f(y) - f(x) = d _ f (x)  (y-x) +  \alpha (y-x),\quad y\in B_X(x,r),
\end{equation*}
where  $\alpha:B_X (0,r)\rightarrow Y$ satisfies $\lim _{h\rightarrow 0}{\alpha (h)}/{\|h\|}=0$. We denote  by
\begin{equation*}
\|d_f (x)\| = \sup _{\zeta\in\partial B_X} \|d_f(x)\zeta\|
\end{equation*}
the  operator  norm of  $d_f(x): X\rightarrow Y$.

For differential   calculus of  mappings on domains in a  normed space  we refer to  Cartan's book \cite{CARTAN.BOOK}.

Joci\'{c}  \cite{JOCIC.JFA} proved a result which is given  in the next  proposition.

\begin{proposition}[see \cite{JOCIC.JFA}]\label{PR.JOCIC}
Let $X$ and $Y$ be Banach spaces, and let $f:{B}_X\rightarrow Y$  be  a Fr\'{e}chet differentiable  mapping. For  a nonconstant operator monotone function
$\varphi$   on  $(-1,1)$, such that $\varphi'$ is  increasing on  $[0,1)$,   the Bloch  type semi-norm
\begin{equation*}
\|f\|_{\mathfrak {B}, \varphi'}  = \sup_{x\in  {B}_X} \frac {\|d_f(x)\|}{\varphi'(\|x\|)}
\end{equation*}
may be expressed as
\begin{equation*}
\|f\|_{\mathfrak {B},\varphi'}=\sup_{(x,y)\in\Delta^c_{B_X}}\frac {\|f(x) - f(y )\|}{\sqrt{ \varphi'(\|x\|)} \|x-y\| \sqrt{\varphi'(\|y\|) }}.
\end{equation*}
\end{proposition}

As in \cite{JOCIC.JFA}  take
\begin{equation*}
\varphi (t)= \tanh^{-1}  t =  \frac12 \log \frac {1+t}{1-t},\quad  t\in (-1,1),
\end{equation*}
in Proposition \ref{PR.JOCIC} in order to recover  Proposition \ref{PR.PAVLOVIC} in  this new and more  general setting. Indeed, the function $\varphi (t)$
is operator  monotone on  $(-1,1)$  (for this fact see  \cite{JOCIC.JFA}) and we  have
\begin{equation*}
\varphi'(t)=(1-t^2)^{-1},\quad   t\in [0,1).
\end{equation*}
It follows that the statement of   Proposition \ref{PR.PAVLOVIC}   is  valid for  Fr\'{e}chet differentiable  mappings on the unit ball  of  a Banach space
with values in an  another Banach space. Note that Joci\'{c} has   removed the  condition concerning continuity of a  differential. In this paper we  prove
a generalization of  Proposition  \ref{PR.JOCIC}  for Fr\'{e}chet    differentiable  mappings on a convex  domain in a  normed space that have  image in an
another normed  space. The proof is based  on our approach. Thus, we are able to  remove  the  completeness    condition  which  figures in the proposition
above, as well as   the condition concerning  continuity  of a differential  on  a domain.

Among     other results in this paper  we obtain the following one. Let $\omega(x)$ be a continuous and positive function on a domain $\Omega$  of a normed
space $X$,    and  let it be  a monotone function of $\|x\|$. If  $f:\Omega \rightarrow Y$   is  a Fr\'{e}chet differentiable mapping which has an image in
a   normed space $Y$, then the   Bloch type   semi-norm
\begin{equation*}
\|f\|_{\mathfrak{B},\omega} = \sup_{x \in \Omega} \frac { \|d_f(x)\| } {\omega(x)}
\end{equation*}
is equal to  the  differential-free semi-norm
\begin{equation*}
\|f\|_{\mathfrak{B},\omega} =   \sup_{(x,y)\in \Delta^c_\Omega}     \min \left\{ \frac{1}{\omega (x)}, \frac {1}{\omega(y)}\right\}
\frac {\|f(x) - f(y)\|}{\|x - y\|}.
\end{equation*}

It seem that we have a better equivalent for the quantity   $\|f\|_{\mathfrak {B},\varphi'}$ then the one given  in  Proposition \ref{PR.JOCIC}, since
\begin{equation*}
\min\left\{\frac{1}{\varphi'(\|x\|)},\frac {1}{\varphi'(\|y\|)}\right\}
\le \frac {1}{{ \varphi'(\|x\|)^{\frac 12} } {\varphi'(\|y\|)^{\frac 12} }}, \quad  x,\,  y\in \Omega.
\end{equation*}
The above inequality is a consequence of the obvious one $\min\{a,b\}\le \sqrt{ab}$ for  nonnegative numbers $a$    and $b$. Note that  we need not the
assumption on  operator  monotonicity of  the  primitive   function for  $\omega$.

The result we have just described  is a  consequence of a more general one given in Theorem \ref{TH.MIN.MAX}. This min-max theorem  also generalize the
author recent result \cite[Corollary 3.1]{MARKOVIC.CMB} for continuously  differentiable mappings  on a  convex domain  in  $\mathbf{R}^m$.

Proposition \ref{PR.PAVLOVIC} is motivated by the Holland-Walsh characterisation of analytic  Bloch functions on the  unit disk obtained at the end of
eighties \cite[Theorem 3]{HOLLAND.WALSH.MATH.ANN}. This characterization says that  an analytic function $f$ on the unit disk $\mathbf{U}$  is a Bloch
function, i.e.,  $(1-|z|^2) |f'(z)|$ is bounded  on  $\mathbf{U}$, if and only if
\begin{equation*}
{\sqrt{1-|z|^2}\sqrt{1-|w|^2}} \frac {|f(z) -  f(w)|}{|z - w|}
\end{equation*}
is bounded  as a  function of $z\in\mathbf{U}$ and $w\in \mathbf{U}$ for $z\ne w$. This clearly follows from Proposition  \ref{PR.PAVLOVIC} having in
mind that  $\|d_f(z)\| =  |f'(z)|$,  $z\in\mathbf{U}$.

We  provide  a proof of  Proposition \ref{PR.PAVLOVIC}  in the  last section which is based   on our approach.  In that section we also find a new
characterization of normal mappings on the unit disk which states:  An analytic function $f$ is normal on $\mathbf{U}$ if  and only if
\begin{equation*}
\frac{\sqrt{1-|z|^2}\sqrt{1-|w|^2}}{\sqrt{1+|f(z)|^2}\sqrt{1+|f(w)|^2}}  \frac {|f(z)  -f(w)|}{|z - w|}
\end{equation*}
is bounded as a function of  $z\in\mathbf{U}$ and $w\in \mathbf{U}$ for $z\ne w$.

\textbf{Acknowledgment.} I would like to thank the referee of this paper for carefully reading the paper and pointing out several errors  in the first
version of the manuscript.

\section{Preliminaries}
Among other things in this section we will briefly describe  the concept of the  Riemann integral  of a  function over a  rectifiable path in a metric
space.  It seems that one cannot find  any reference for integration in a such general context.  Of course,  if one restricts attention to the domains
in  $\mathbf{R}^n$,  this becomes  an ordinary  material.

A path $\gamma$ in a metric space $X$ is a continuous mapping $\gamma:[a,b]  \rightarrow X$, where $[a,b]\subseteq \mathbf{R}$. It is convenient to
identify the path $\gamma$ with  its image $\gamma([a,b])\subseteq X$.

Let $\mathcal{P}[a,b]$ be the family of  all  partitions of the  segment $[a,b]$, i.e., the family of all finite sequences $T: t_0,t_1,\dots,t_n$, such
that $n\in\mathbf{N}$, and  $a=t_0<t_1<\dots<  t_n=b$.

For a path  $\gamma:[a,b]\rightarrow X$ and  a partition  $T\in \mathcal{P}[a,b]$,   $T: t_0,t_1,\dots,t_n$, denote
\begin{equation*}
\sigma _\gamma  (T) =  \sum_{i=1}^n d  (\gamma(t_{i-1}),\gamma (t_i)).
\end{equation*}
The  path $\gamma$ is  rectifiable if the set $\{\sigma _\gamma (T):  T \in\mathcal{P}[a,b] \}$  is bounded.  The length  of the path  $\gamma$ is
\begin{equation*}
\ell  (\gamma)   = \sup_{  T\in\mathcal {P}[a,b] } \sigma _\gamma  (T).
\end{equation*}
Note that  we  obviously have $\ell (\gamma) \ge  d (\gamma(a),\gamma(b))$.

It is not hard to prove  that if  $\gamma:[a,b]\rightarrow X$ is a rectifiable path in a metric space $X$,  and if $[c,d] \subseteq [a,b]$, then the  restricted path $\gamma:[c,d]\rightarrow X$ is also  rectifiable. This    new  path is denoted  by $\gamma_{[c,d]}$. It is clear that  $\ell (\gamma)
 = \ell  (\gamma_{[a,c]}) + \ell(\gamma_{[c,b]})$.

The segment $[x,y]$ in a normed space $X$ with the  endpoints  $x\in X$ and $y\in X$  is the path  $\gamma(t) = (1-t) x +  ty$, $t\in [0,1]$.  It is
easy to show that $[x,y]$  is a  rectifiable path, and  $\ell ([x,y]) = \|x-y\|$.

Let  $X$ and $Y$ be metric spaces.  For a  mapping $f:X\rightarrow  Y$ introduce
\begin{equation*}
d^\ast_f (x)=\limsup _{y\rightarrow x} \frac {d   ( f(x), f(y))}{d(x,y)},   \quad x\in X.
\end{equation*}
We set $d^\ast _f (x) = 0$,  if $x\in X$ is an isolated point.

The  content of the following lemma is maybe well known, however  we will  provide a proof  since  we are   not able to   find  a  reference.

\begin{lemma}\label{LE.NORMED.DIFF}
Let $\Omega\subseteq X$ be a domain in  a normed space $X$, and  $f:\Omega \rightarrow Y$ be a Fr\'{e}chet differentiable mapping  at  $x\in\Omega$, where
$Y$ is an another normed space. For the operator  norm of the Frech\'{e}t differential   $d_f(x):X\rightarrow Y$, $\|d_ f (x) \|  = \sup_{\zeta\in\partial
 B_X} \| d_f (x) \zeta\|$   we have $\| d_ f (x)\|=  d_f^\ast (x)$,  i.e.,
\begin{equation*}
\| d _ f (x)\|    =  \limsup_{y\rightarrow x} \frac {\|f(x) - f(y)\|}{\|x-y\|}.
\end{equation*}
\end{lemma}

\begin{proof}
For the sake of  simplicity let us denote $d_f(x) = A$. In a sufficiently small open ball $B_X(x,r)$   there exists  a mapping
$\alpha : B_X(x,r)\rightarrow Y$ such that
\begin{equation*}
f(y)-f(x) = A  (y-x) + \alpha (y-x),\quad   \lim_{y\rightarrow x}  \frac{\alpha (y-x)}{\|y-x\|}  = 0.
\end{equation*}
Therefore,
\begin{equation*}\begin{split}
\|f(y)  - f(x)\| & =  \|A (y-x) + \alpha (y-x)  \|    \le \|A (y-x)\| + \|\alpha (y-x)  \| \\& \le \|A\|\|y-x\| +\|\alpha (y-x)\|.
\end{split} \end{equation*}
It follows
\begin{equation*}
\frac {\|f(y)  - f(x)\|} {\|y-x\|} \le      \|A\| + \frac { \|\alpha (y-x)  \|}{\|y-x\|},
\end{equation*}
and
\begin{equation*}
\limsup_{y\rightarrow  x}\frac {\|f(y)  - f(x)\|} {\|y-x\|} \le  \|A\|.
 \end{equation*}

Assume now   contrary,  that  we have   the strict inequality above.  Then there exists a real number   $t$ that satisfies
\begin{equation*}
\limsup_{y\rightarrow  x}  \frac {\|f(y)  - f(x)\|} {\|y-x\|}  < t< \| A\|.
\end{equation*}
This means  that there exists  an  open ball  $B_X(x,\rho )$  such that
\begin{equation*}
\|f(y ) - f(x)\|  \le  t \|y-x\|,\quad   y\in B_X(x,\rho ).
\end{equation*}
For $s\in(0,\rho)$ and $\zeta\in \partial B_X$, we have $y =x +s \zeta \in B_X (x,\rho)$. Since $f(y) - f(x) = A(y-x)+ \alpha (y-x) $, the  preceding
inequality   becomes  $\|A(s\zeta  ) + \alpha (s\zeta) \| \le t\|s\zeta\|$, i.e.,  $\|A\zeta  +{\alpha (s\zeta)}/{s}\| \le t$. If  we  let
$s\rightarrow 0$ we obtain $\| A\zeta  \|\le t$. Since this holds for every $\zeta\in \partial  B_X$, we reach a  contradiction  $\|A\|\le  t$.
\end{proof}

The   following   lema may be found in \cite{C.J.CJM},                 but for the sake of  completeness we will adapt  a proof  here.

\begin{lemma}[see \cite{C.J.CJM}]\label{LE.IMAGE}
Let $X$ and $Y$ be metric spaces,     let $\gamma$ be a rectifiable path, and let a mapping $f: X \rightarrow Y$ be a such one   that there exists
$d^\ast_f (x)$ for every   $x\in X$.  Assume,   moreover,  that there exists a  constant $m$ such that  $d^\ast_f(x) \le m$,  $x\in X$. Then  the
path $\delta = f \circ \gamma $ is also  rectifiable, and  $\ell (\delta ) \le m  \ell(\gamma)$.
\end{lemma}

\begin{proof}
Let  $\varepsilon> 0$ be any number, and let $\gamma:[a,b]\rightarrow X$.  We will firstly prove  that
\begin{equation*}
d( f (\gamma (t)), f (\gamma (s))) \le   (m+\varepsilon) \ell (\gamma|_ {[s, t]}), \quad  a\le  s <  t \le  b
\end{equation*}
Indeed, for every $x \in  [s, t]$  we have $d^\ast_f (x) \le  m$, so there exists  an interval $I_ x = (x - \delta_x , x + \delta_x)$, $\delta_x>0$,
such that
\begin{equation*}
d( f (\gamma(x)), f (\gamma(x'))) \le  (m  + \varepsilon) d(\gamma(x), \gamma(x')),\quad  x' \in  I _x\cap  [s, t].
\end{equation*}
Let $J$ be a finite subset of $[s, t]$ such that $\{I_x: x\in J\}$ is a cover of $[s,t]$ containing no proper subcover. Let $x_1,x_3,\dots, x_{2n-1}$
be a sequence    of  elements of  $J$ in the increasing   order. Since of   minimality of the cover, for every $i$, $0\le i\le  n-1$ there exists  $x_{2i}\in  I_{2i-1}\cap  I_{2i+1}\cap (x_{2i-1}, x_{2i+1})$.   Let, moreover, $x_0  = s \in  I_{x_1} $ and $x_{2n}    =  t \in  I_{x_{2n-1}}$. In
particular, we have $s = x_0 < x_1 < x_2 <\dots < x_{2n-1}< x_{2n}= t$  and $x_{2i},x_{2i+2}\in  I_{2i+1}$ for every  $i \in \{0,1,\dots,  n-1\}$.     Therefore,
\begin{equation*}\begin{split}
d ( f (\gamma (t)), f (\gamma(s)) ) & \le \sum _{ k = 0}^{2n-1} d ( f (\gamma(x_k)),   f (\gamma(x_{k+1}))
\\& \le\sum _{k = 0}^{2n-1}(m  + \varepsilon)d  ( \gamma(x_k),  \gamma(x_{k+1})
\\& \le  (m  + \varepsilon)\ell(\gamma|_ {[s, t]}),
\end{split}\end{equation*}
justifying the inequality stated at the beginning of the proof.

To  finish   the proof, choose a partition  $T: t_0,t_1,\dots,t_n $ of  $[a,b]$ such that
\begin{equation*}
\ell(  f \circ   \gamma) \le \sum_{i=0}^{n-1} d( f (\gamma(t_{i+1})), f (\gamma(t_i))) +   \varepsilon.
\end{equation*}
By the just   proved inequality we obtain
\begin{equation*}\begin{split}
\ell( f \circ  \gamma)   & \le  \sum _ {i=0}^{n-1} d ( f (\gamma(t_{i+1})), f (\gamma(t_i)) +  \varepsilon
\\&\le \sum _{i=0}^{n-1}(m  + \varepsilon) \ell (\gamma|_{[t_{i-1} , t_i ]}) + \varepsilon
\\& = (m  + \varepsilon) \ell(\gamma) + \varepsilon.
\end{split}\end{equation*}
Since  this holds for every   $\varepsilon>0$,  the desired inequality,    $\ell (f\circ \gamma ) \le m \ell (\gamma)$,   follows.
\end{proof}

Let $\gamma:[a,b]\rightarrow X$ be  a rectifiable path in a metric space $X$,   and let $f$ be  a function on $X$. To a  partition
$T:t_0,t_1,\dots,t_n$  of a segment $[a,b]$ coordinate an  another sequence $S:s_1,s_2,\dots,s_n$, such that $s_i\in  [t_{i-1},t_i ]$ for every
$i$, $1\le i \le n$,  and  denote
\begin{equation*}
\sigma_\gamma  (f,T,S)
= \sum _{i=1}^n f (\gamma (s_i)) \ell(\gamma|_{[t_{i-1},t_i]} ).
\end{equation*}

The Riemann integral $\int_\gamma  f$ is a number (if there  exists a such number) which    satisfies:   for   every  $\varepsilon>0$
there  exists $\delta>0$ such that
\begin{equation*}
\left|\int_\gamma  f   -  \sigma_\gamma (f,T,S) \right|<\varepsilon
\end{equation*}
for every $T\in \mathcal {P}[a,b]$,       $\mathrm {diam}(T)  = \max_{1\le i\le n} |t_i - t_{i-1}|< \delta$  and  every sequence  $S$
coordinated   to $T$.

It  is not hard to show that instead of $\sigma_\gamma  (f,T,S)$ we could  consider
\begin{equation*}
\tilde{\sigma}_\gamma  (f,T,S)
 = \sum _{i=1}^n f (\gamma (s_i)) d(\gamma (t_{i-1}, \gamma (t_i) ),
\end{equation*}
in order to define $\int_\gamma f $.

It is a routine to prove that the following statements are equivalent (actually, the  proof  is almost  the same as in the Euclidean case):

(i) There exists $\int_\gamma f$.

(ii) For every $\epsilon>0$ there exists $\delta>0$ such that $|\sigma_\gamma(f,T,S)-\sigma_\gamma (f,T',S')|  <\varepsilon$ for every partitions $T$, $T'$
such that $\mathrm {diam}(T)<\delta$, $\mathrm {diam}(T')<\delta$, and every sequences $S$   coordinated  to $T$,  and $S'$ coordinated to $T'$.

(iii) For every  $\epsilon>0$ there exist $\delta>0$ such that   $|\sigma_\gamma (f,T,S) -\sigma_\gamma (f,T,S')|  <\varepsilon$ for every partition $T$,
$\mathrm{diam} (T)<\delta$, and every sequences $S$ and $S'$ coordinated to $T$.

If $f$ is a continuous function on a metric space $X$, and if a path $\gamma : [a,b] \rightarrow X$ is rectifiable, then the Riemann integral
$\int_ \gamma f$  exists. Indeed, we will use the third criterion stated above. Let   $\varepsilon>0$ be any number.  Since of uniform continuity for
$f\circ \gamma$   on    $[a,b]$, there exists $\delta>0$ such that $|f\circ\gamma(s)  -f \circ\gamma(s')| <\varepsilon$  if   $|s - s'|< \delta$. For any
partition  $T:  t_0,t_1,\dots,t_n$  of $[a,b]$ with $\mathrm {diam} (T)< \delta$, and sequences  $S:s_1,s_2,\dots,s_n$  and $S':s'_1, s'_2,\dots,s'_n$
coordinated to $T$,  we have
\begin{equation*}\begin{split}
|\sigma(f,T,S)  - \sigma(f,T,S')|  &   = \left|\sum _{i=1}^n f (\gamma (s_i))
 \ell( \gamma|_{ [t_{i-1},t_i] } )-\sum _{i=1}^n f (\gamma ({s}'_i))
\ell( \gamma|_{ [t_{i-1},t_i] }  ) \right|
\\& \le  \sum _{i=1}^n |f (\gamma (s_i)) - f (\gamma ({s}'_i)) |
 \ell( \gamma|_{ [t_{i-1},t_i] } )
\\&   \le\varepsilon  \sum _{i=1}^n    \ell( \gamma|_{ [t_{i-1},t_i] } ) \le \varepsilon \ell(\gamma).
\end{split}\end{equation*}

It is  clear  that if $f_1\le f_2$, then  $\int_\gamma f_1\le \int_\gamma f_2$ provided  that the both integrals exist.

If there exists  $\int_\gamma f$, then we have    $\int_\gamma f = \lim_{n\rightarrow\infty} \sigma_\gamma (f,T_n, S_n)$,  where $(T_n)_{n\in\mathbf{N}}$
is   any  sequence  of partitions of $[a,b]$  such that  $\mathrm {diam} (T_n) \rightarrow 0$, as $n\rightarrow \infty$, and  $S_n$ is a finite sequence
coordinated to  $T_n$ for every $n\in\mathbf{N}$.

If  $[x,y]$  is  a  segment  in a normed space $X$, then   we have
\begin{equation*}
\int _{[x,y]} f= \|x- y\|_X  \int_0^1\tilde{f} (t)dt,\quad \tilde{f} (t) = f ((1-t) x+t y),\,  t\in [0,1],
\end{equation*}
where on the right side is  the  ordinary Riemann integral. Indeed,
\begin{equation*}\begin{split}
&\int _{[x,y]} f
\\&  =  \lim_{\mathrm {diam}(T)\rightarrow 0} \sum_{i=1}^n  f ((1-t_{i-1}) x+t_{i-1} y) \| ((1-t_{i}) x+t_{i} y  ) - ((1-t_{i-1}) x+t_{i-1} y)  \|
\\& =\|x- y \| \lim_{\mathrm {diam}(T)\rightarrow  0} \sum _{i=1}^n f ((1-t_{i-1}) x+t_{i-1} y) | t_i - t_{i-1}|
\\&= \|x- y\|  \int_0^1\tilde{f} (t)dt.
\end{split}\end{equation*}

A metric space $X$ is rectifiable  connected if  for every $x\in X$ and $y\in X$  there  exists  a rectifiable  path  $\gamma: [a,b]\rightarrow X$ that
connects $x$ and $y$, i.e., a such one that  $\gamma ( a)=x$ and $\gamma (b) =  y$.  Note that any  domain in a normed space is rectifiable connected
metric space.

An  everywhere positive and continuous function $\omega$  on a metric space $X$   will  be  called  a  weight  function. If $\gamma$ is a rectifiable path,
its $\omega$-length  is    given by the  Riemann  integral
\begin{equation*}
\ell _ \omega  (\gamma) = \int_\gamma \omega.
\end{equation*}
Note that  in particulary    we  have
\begin{equation*}
\ell _1 (\gamma) = \int_\gamma 1= \ell(\gamma).
\end{equation*}

The $\omega$-distance between     $x\in X$ and $y\in X$  (on the metric  space $X$)  is
\begin{equation*}
d_\omega    (x, y)  =  \inf_\gamma \ell_\omega (\gamma),
\end{equation*}
where  $\gamma$  is among all rectifiable paths  that  connect  $x$ and $y$. It is not hard to  check that $d_\omega$ is a distance  function on $X$.

\section{The equality of the  Bloch and  the  Lipschitz number}
The aim of the  present section is to put the equality statement of Proposition \ref{PR.PAVLOVIC} in a very general setting. We will consider two types of
growth conditions for mappings between metric spaces, and  we will  introduce   the    Bloch number and the Lipschitz number of a mapping. The generalized
result of  Pavlovi\'{c}, given in  Theorem \ref{TH.PAVLOVIC.MS},    says  that these numbers  are equal.

Let us first introduce the     Bloch type  growth  condition and the Bloch  number of a mapping. Let  $\omega$ and $\tilde{\omega}$ be weight functions on
metric  spaces $X$ and $Y$, respectively. Consider a mapping $f :X\rightarrow  Y$  such  that $d^\ast _ f(x)$ is finite for every  $x\in X$, and  which
satisfies  the  growth   condition of the following form  ${\tilde{\omega}  (f( x ))} d^\ast _ f(x)\le  m {\omega(x)}$ for  every   $x\in  X$, where $m$ is
a positive number.  Introduce       the following quantity
\begin{equation*}
\mathfrak  {B}_f = \sup_{x\in X} \frac  {\tilde{\omega} (f(x))}{\omega(x)} d^\ast _f(x),
\end{equation*}
which is  called  the Bloch number  of the mapping $f$  with respect to the weights $\omega$ and $\tilde{\omega}$. Note that if $f$   satisfies the Bloch
type growth  condition, then $d^\ast_f$ is bounded on compact subsets of $X$. This follows since  $\omega$  and   $\tilde{\omega}$ are continuous.

Introduce now the  Lipschitz   type growth condition and the Lipschitz  number of a mapping between metric spaces. Here  we consider mappings
$f:X\rightarrow Y $   that satisfy the growth  condition   of the following form  $d (f(x),f(y))\le m(x,y) d  (x,y)$, $x,\, y\in X$  for  a positive
function $m(x,y)$ on $X\times X$.

For  weight  functions    $\omega$ on  $X$, and $\tilde{\omega}$ on $Y$,  and a mapping $f:X\rightarrow   Y$  consider an everywhere
positive function   $\Psi_f $ on  $X\times   X$  such  that it   satisfies the  following  conditions:

\noindent (1)
\begin{equation*}
\Psi_f (x,y)  =   \Psi_f (y,x),\quad   x,\,  y\in X;
\end{equation*}

\noindent (2)
\begin{equation*}
\Psi_f (x,x) = \frac{\tilde{\omega} (f(x))} {\omega(x)},\quad    x\in X;
\end{equation*}

\noindent (3)
\begin{equation*}
\liminf_{y\rightarrow x} \Psi_f(x,y)\ge \Psi_f (x,x),\quad   x\in X;
\end{equation*}

\noindent (4)
 \begin{equation*}\Psi_f(x,y)  \frac{d(f(x), f(y))}{d(x,y)}
\le  \frac {d_{\tilde{\omega}}(f(x),f(y))}   {d_\omega(x,y)},\quad  (x,y)\in \Delta_X^c.
\end{equation*}

\noindent  We say that $\Psi_f$ is an admissible for the mapping $f$ with respect to the weight functions $\omega$ and $\tilde{\omega}$. Note that if
$\Psi_f(x,y)$ is not symmetric,   but satisfies all other conditions stated above, we  can replace it  by the  symmetric function
\begin{equation*}
\tilde{\Psi}_f  (x,y) = \max\{ \Psi_f(x,y),\Psi_f(y,x)\},\quad   x,\, y\in X.
\end{equation*}
This new function will  be  also admissible for $f$,  as it is  easy to  check.

Let us   note that if    we set $\tilde{\omega}\equiv 1$, then the distance $d_{\tilde{\omega}}$   is   in some cases equal to   the metric on $Y$. For
example, this  occurs in   normed spaces.  In this case  the fourth condition  of  admissibility is independent of the function $f$, and one has to find
an universal admissible function (that depend only on the weights) $\Psi(x,y)= \Psi_f(x,y)$  which satisfies the following  simplified condition

\noindent (4')
\begin{equation*}
\Psi_f(x,y) {d_\omega(x,y)}\le  d (x,y),\quad   x,\, y\in X.
\end{equation*}

Introduce  now  the  following    quantity
\begin{equation*}
\mathfrak{L}_f  = \sup _{(x,y)\in \Delta_X^c }  \Psi_f(x,y)\frac {d(f(x),f(y))}{d(x,y)}
\end{equation*}
where $\Psi_f$  is any admissible for  $f$ with respect  to $\omega $ and $\tilde{\omega}$. We call it the Lipschitz number of  $f$. The result stated
below  says  that the Lipschitz number  does not depend on the choice of an admissible  function for mappings $f$. Therefore,  the definition of the
Lipschitz number  $\mathfrak{L}_f$    is   correct.

We will now prove our main result in this section.

\begin{theorem}\label{TH.PAVLOVIC.MS}
Let $X$ and $Y$ be rectifiable  connected  metrics spaces with weight functions $\omega$ and $\tilde{\omega}$, respectively. Let  $f:X \rightarrow Y$
be a mapping such that for every $x\in X$ there exists $d^\ast _f (x)$, and let $\Psi_f$  be an  admissible  function for the mapping $f$ with respect
to the weights $\omega$ and $\tilde{\omega}$. If the Bloch number  $\mathfrak{B}_f$,  or the  Lipschitz number $\mathfrak{L}_f$,  is  finite, then the
both  numbers are  finite,   and $\mathfrak{B}_f  =  \mathfrak{L}_f$. Therefore,  the Lipschitz number  is  independent of the choice of an admissible
function $\Psi_f$,  and it may   be   expressed in the differential-free way
\begin{equation*}
\mathfrak{B}_f = \sup_{(x,y)\in\Delta_X^c} \Psi_f (x,y )   \frac{d(f(x),f(y))}{d( x,y)}.
\end{equation*}
\end{theorem}

\begin{proof}
For one direction, assume that the Lipschitz number of  a function $f$ is finite, i.e.,  assume  that
\begin{equation*}
\mathfrak {L}_f   =   \sup _{(x,y)\in  \Delta_X^c } \Psi_f(x,y) \frac {d (f(x),f(y)) }{d(x,y)}
\end{equation*}
is finite. We are going to show that $\mathfrak {B}_f\le \mathfrak {L}_f$,  which  implies that the Bloch number $\mathfrak{B}_f$ is  also finite.

If $A (\zeta)$  and  $B (\zeta)$ are nonnegative functions on  a metric space,  then we have
\begin{equation*}
\limsup _{\zeta\rightarrow  \eta } A (\zeta) B (\zeta)  \ge  \liminf_{\zeta\rightarrow  \eta }   A (\zeta)
\limsup_{\zeta\rightarrow  \eta } B(\zeta).
\end{equation*}
We will use this fact below.

For every  $x\in X$ we have
\begin{equation*}\begin{split}
\mathfrak{L}_f&  =  \sup _{   (y,  z)\in \Delta_X^c }  \Psi_f(y,z)\frac {d(f(y), f(z))}{d(y,z)}\ge
\limsup_{z \rightarrow x} \Psi_f(x,z) \frac {d (f(x) , f(z))}{d(x,z)}
\\& \ge \liminf_{z\rightarrow x} \Psi_f(x,z) \limsup_{z\rightarrow x}   \frac {d (f(x) , f(z))}{d(x,z)}  \ge \Psi_f(x,x) d^\ast_f(x)
 = \frac {\tilde { \omega}(f(x))}{\omega(x)}  d^\ast_f(x).
\end{split}\end{equation*}
It follows that
\begin{equation*}
\mathfrak{L}_f \ge\sup_{x\in X}   \frac {\tilde { \omega}(f(x))}{\omega(x)}  d^\ast_f(x) = \mathfrak{B}_f,
\end{equation*}
which we aimed to prove.

Assume now  that the  Bloch number  $\mathfrak{B}_f$  of a mapping $f$ is finite,  and form the sake of  simplicity  let $\mathfrak{B}_f = m $,  i.e., let
$\tilde{\omega}(f(x))d^\ast_ f (x) \le m\omega(x)$,   $x\in X$. We will first  prove  that in this case we  have
\begin{equation*}
d_{\tilde{\omega}} (f(x),f(y)) \le m  d_\omega (x, y),\quad  x,\, y\in X.
\end{equation*}
Let $\gamma:[a,b] \rightarrow  X$ be any rectifiable path connecting $x\in X$ and $y\in X$, i.e.,  such that $\gamma(a)= x$ and $\gamma(b) = y $. Consider
the path $\delta = f \circ \gamma\subseteq Y$. This  path  connects $f(x)$ and $f(y)$. Since $d^\ast _f(x)$ is finite for every $x\in X$,
by Lemma \ref{LE.IMAGE}, the  path $\delta = f\circ\gamma$  is   also  a rectifiable  one. Therefore, $\int _ \delta {\tilde{\omega}}$ is a finite number.

Let  $\varepsilon> 0$ be arbitrary. Since the function
\begin{equation*}
g(z) =\frac  {m  \omega(z)}{\tilde{\omega} (f(z))} ,\quad  z\in X
\end{equation*}
is  continuous on $X$,  and since any  path  is a compact set, we may found a partition $T:t_0,t_1,\dots,t_n$ of the segment $[a,b]$ such  that
$\mathrm {diam}  (T) <\varepsilon$,  and
\begin{equation*}
|g(\gamma (t)) - g(\gamma (t_{i-1} ) )|<\varepsilon, \quad t \in [t_{i-1},t_{i}],\,   1\le  i\le  n.
\end{equation*}
Since  $g(\gamma (t)) < g(\gamma (t_{i-1}))+\varepsilon$, it follows
\begin{equation*}
d^\ast _f (\gamma (t))\le  g (\gamma (t)) < g(\gamma (t_{i-1}))+\varepsilon, \quad  t\in [t_{i-1},t_i], \, 1\le i\le n.
\end{equation*}
Let $S$ be the sequence  $t_0,t_1,\dots, t_{n-1}$. Since $\gamma$   is compact, there   exists a constant $C$ such that $\tilde{\omega} (f ( x)) \le  C$,
$x \in  \gamma$. Now,   we have
\begin{equation*}\begin{split}
&\sigma _ {\delta}  (\tilde {\omega}, T, S)\\& = \sum _ {i=1}^n  \tilde { \omega}(f(\gamma( t_{i-1} )))\ell(\delta|_{[t_{i-1},t_i]})
\le\sum_{i=1}^n\tilde { \omega}(f(\gamma( t_{i-1} ))) (g(\gamma (t_{i-1}))+\varepsilon)  \ell(\gamma|_{[t_{i-1},t_i]})
\\& \le \sum _ {i=1}^n  \tilde { \omega}(f(\gamma( t_{i-1} )))  g(\gamma (t_{i-1})) \ell(\gamma|_{[t_{i-1},t_i]})
+ \varepsilon \sum _ {i=1}^n  \tilde { \omega}(f(\gamma( t_{i-1} )))  \ell (\gamma|_{[t_{i-1},t_i]})
\\& \le \sum _ {i=1}^n  \tilde { \omega}(f(\gamma( t_{i-1} )))
 \frac  {m  \omega(\gamma( t_{i-1} ))}{\tilde{\omega} (f(\gamma( t_{i-1} )))}
 \ell(\gamma|_{[t_{i-1},t_i]})
+ \varepsilon  C \sum _ {i=1}^n    \ell(\gamma|_{[t_{i-1},t_i]})
\\& \le m \sum _ {i=1}^n   {\omega(\gamma( t_{i-1} ))} \ell(\gamma|_{[t_{i-1},t_i]})
+ \varepsilon  C    \ell(\gamma)
=   m \sigma_\gamma(\omega, T,S)  + \varepsilon  C     \ell(\gamma)
\end{split}\end{equation*}
If we let   $\varepsilon \rightarrow  0$ above, we obtain $\int_\delta \tilde {\omega} \le m \int_\gamma \omega$.  Now,  if we  take the infimum over all
paths  $\gamma$,  we obtain $\inf_{\gamma} \int_{f\circ \gamma} {\tilde{\omega}} \le  m  d_ \omega  (x,y)$.  If $\delta$  is among all paths  that
connect   $f(x)$ and $f(y)$, we have
\begin{equation*}
d_{\tilde{\omega}} (f(x),f(y)) = \inf _\delta \int _ \delta {\tilde{\omega}} \le \inf_\gamma \int_{f\circ \gamma}{\tilde{\omega}}\le m d_\omega (x, y).
\end{equation*}
Therefore,   we  have proved that aimed inequality.

We will prove now the reverse inequality  $\mathfrak{L}_f\le\mathfrak{B}_f$, which in particular implies  that  the Lipschitz number $\mathfrak{L}_f$ is
also  finite.  By the  above we have ${d_{\tilde{\omega}} (f(x),  f(y))}\le \mathfrak{B}_f d_{\omega} (x,y)$ for  every  $x\in X$ and  $y\in X$. Applying
now   conditions  posed     on     the  admissible  function  $\Psi_f$,   we  obtain
\begin{equation*}
\Psi_f(x,y)\frac {d(f(x),f(y)) }{d (x,y)} \le \frac {d_{\tilde{\omega}} (f(x), f(y))}{d_{\omega}(x,y)}\le  \mathfrak{B}_f, \quad (x,y)\in  \Delta_X^c.
\end{equation*}
It follows that
\begin{equation*}
\mathfrak{L}_f  =  \sup_{(x,y)\in \Delta_X^c} \Psi_f(x,y)\frac {d (f(x),f(y))}{d(x,y)}  \le \mathfrak{B}_f,
\end{equation*}
which we needed  to prove.
\end{proof}

In order to apply Theorem \ref{TH.PAVLOVIC.MS} for normed spaces, we will need the following auxiliary result.

\begin{lemma}\label{LE.LIM}
Let $X$ be a  normed space, let $\Omega\subseteq X$ be a  domain, and let $\omega$ a weight function on $\Omega$. For $\omega$-distance $d_\omega$ on
$\Omega$,   we have
\begin{equation*}
\lim _{y\rightarrow  x}\frac { d_\omega(x,y)} {\|x-y\|}  = \omega (x), \quad x\in  \Omega.
\end{equation*}
\end{lemma}

\begin{proof}
Since  $\omega $ is continuous, there exists an open ball $B_X(x,r)\subseteq \Omega$ such that
\begin{equation*}
0<\omega(x)-\varepsilon<\omega (y)<\omega(x)+\varepsilon,\quad y\in B_X(x,r),
\end{equation*}
where $\varepsilon> 0$  is  a  sufficiently small   number. Now,  we have
\begin{equation*}
d_\omega (x,y)  \le \int_{[x,y]}  \omega  \le   (\omega (x) +\varepsilon )\ell ([x,y] ) =  (\omega (x) +\varepsilon ) \|x-y\|.
\end{equation*}
On the other hand, if $\gamma\subseteq \Omega$ is among paths that connect   $x$  and $y$, then
\begin{equation*}
d_{\omega} (x,y) = \inf_\gamma \int_\gamma \omega\ge (\omega (x) - \varepsilon)\|x-y\|.
\end{equation*}
Therefore,
\begin{equation*}
\omega(x) - \varepsilon  \le \frac{d_\omega (x,y)}{\|x-y\|} \le \omega(x) + \varepsilon,\quad    y\in B_X(x,r).
\end{equation*}
This means that
\begin{equation*}
\lim _{y\rightarrow  x}\frac { d_\omega(x,y)} {d (x,y)} = \omega (x).
\end{equation*}
\end{proof}

By Lemma \ref{LE.LIM} we have that the topology on $\Omega $ induced by the $\omega$-distance $d_\omega$ is the same as the topology
induced by the norm on $X$.  For similar results we  refer to \cite{ZHU.JLMS}.

\begin{corollary}\label{COR.M}   Let $X$ and $Y$ be normed spaces, and  let  $\omega$   and $\tilde{\omega}$  be weight functions on
domains  $\Omega\subseteq X$ and $\tilde{\Omega}\subseteq Y$,  respectively. The Bloch number
\begin{equation*}
\mathfrak{B}_f      =   \sup_{x\in \Omega} \frac {\tilde{\omega} (f(x))} {\omega (x)} d^\ast_ f(x)
\end{equation*}
of a mapping $f:\Omega\rightarrow\tilde{\Omega}$  for which        $d_f^\ast(x)$  is finite for every   $x\in \Omega$,   is equal to
\begin{equation*}
\mathfrak{B}_f   = \sup_{(x,y)\in \Delta^c_\Omega}  \frac{d_{\tilde{\omega}} (f(x), f(y))} {d_\omega  (x,y)}.
\end{equation*}
\end{corollary}

\begin{proof}
Indeed,  an  admissible   function $\Psi _f$   is given by
\begin{equation*}
\Psi  _ f(x,y) =
\begin{cases}
\frac{ d_{\tilde{\omega}}(f(x),f(y))}{d(f(x), f(y))}/
\frac {d_\omega(x,y)}{d (x,y)}, & \mbox{if}\ x\ne y,f(x)\ne f(y); \\
 \tilde{\omega} (f(x)) /   \frac {d_\omega(x,y)}{d (x,y)}, & \mbox{if}\ x\ne y, f(x)=f(y) ;\\
{\tilde{\omega} (f(x))} /{\omega(x)},  & \mbox{if}\ x=y,f(x)=f(y).
\end{cases}
\end{equation*}
Having in mind Lemma \ref{LE.LIM} it follows that
\begin{equation*}
\liminf_{y\rightarrow x} \Psi_f (x,y) = \lim_{y\rightarrow x} \Psi_f (x,y) =\frac{\tilde{\omega} (f(x))} {\omega(x)} =  \Psi _f(x,x).
\end{equation*}
Other           three admissability    conditions for $\Psi_f$ are obviously satisfied.
\end{proof}

In other words, the corollary says that  a  mapping $f:X\rightarrow Y$ satisfies the  condition   $\tilde {\omega} (f(x)) d^\ast_f (x)\le m
\omega (x)$,    $x\in X$, where  $m$ is  a  constant,   if and only if  $d_{\tilde {\omega} } (f(x),f(y))\le m  d_\omega (x,y)$, $x,y\in X$.

\section{Admissible functions in some special cases}
Having in mind  Lemma \ref{LE.NORMED.DIFF}, as an immediate consequence of Theorem \ref{TH.PAVLOVIC.MS} (and  special case of Corollary \ref{COR.M}) we
have

\begin{proposition}
Let $\omega$  be a  weight function on   a  domain   $\Omega$  in  the  normed space $X$ and let $d_\omega$ be the $\omega$-distance on $\Omega$. For a
Fr\'{e}chet  differentiable  mappings  $f:\Omega\rightarrow Y$. The Bloch type semi-norm
\begin{equation*}
\|f\|_{\mathfrak{B}, \omega}  =    \sup_{x\in \Omega} \frac {\|d _ f(x)\|} {\omega (x) }
\end{equation*}
is equal to
\begin{equation*}
\|f\|_{\mathfrak{B}, \omega}  = \sup_{(x,y)\in \Delta^c_\Omega}  \frac{\|f(x) - f(y)\|_Y} {d_\omega  (x,y)}
\end{equation*}
\end{proposition}

Let us note that an  admissible  function  here is   independent on $f$, and  is given by
\begin{equation*}
\Psi   (x,y) =
\begin{cases}
{\|x - y \|}/ {d_\omega(x,y)}, & \mbox{if}\ x\ne y ;\\
1 /{\omega(x)},  & \mbox{if}\ x=y.
\end{cases}
\end{equation*}

Since the explicit expression for the   distance   function $d_\omega$ is  difficult to find in general, the aim of the rest of this section  is to find
simple  admissible  functions in some   special cases. We do that for  weights that are derivative of an  operator monotone function on an interval  and
in the case when   weights are monotone in norm.

\subsection{The case of weights that are derivatives of  operator   monotone functions}
We will start this subsection with preliminaries on operator monotone function on an open interval $I\subseteq\mathbf{R}$. The class of operator monotone
functions on the interval $I$, denoted by $\mathrm {OM}(I)$, contains all real-valued functions $\varphi$ on $I$  that preserve the Hilbert space operator
ordering in the following sense: let $A$ and $B$  be  self-adjoint operators  on a Hilbert space  with their  spectra contained  in the interval  $I$, and
satisfying  $A \le  B$, then $\varphi(A)\le\varphi(B)$.         For  a more in  operator monotone  functions we refer to the  classical reference book
\cite{DONOGHUE.BOOK},  and to  the  recent  one \cite{HIAI.INTER}. We point out that  an operator monotone function on  $I$ is continuously differentiable
on $I$, and $\varphi '(t)>0$ for $t\in I$ unless it is a constant.

An analytic function $\varphi:\{z\in \mathbf{C}:\Im z>0\}\rightarrow\{z\in\mathbf{C}:\Im z\ge0\}$, belongs  to the Pick  class $\mathrm {P}(I)$ if it may
be analytically extended  across the interval $I$ by the Schwarz reflection principle, i.e., $\varphi (z ) = \overline  { \varphi (\overline {z}) }$. It
is clear  that   $\varphi (t)\in  \mathbf{R}$ if $t\in \mathbf{R}$. By the   Nevanlinna theorem  a  function   $\varphi\in \mathrm {P}(I)$ may be
represented   as
\begin{equation*}
\varphi (z) = \alpha +\beta z +  \int_{\mathbf{R}\backslash I} \frac{1+\lambda z}  { \lambda - z} d\mu (\lambda),\quad  z\in \mathbf{C}\backslash
\mathbf{R}\cup I,
\end{equation*}
where  $\mu$ is finite  positive  Borel  measure on $\mathbf{R}\backslash I$, and  $\alpha$,  $\beta\in \mathbf{R}$,  $\beta\ge  0$. Moreover, it may be
showed that a  function   $\varphi \in  \mathrm {P} (-1,1)$ may be represented in the  way
\begin{equation*}
\varphi (z)   = \varphi (0) +\varphi' (0) \int_{ [-1,1]} \frac {z}{1-tz } d\mu(t),\quad z\in  \mathbf{C} \backslash\mathbf{R}\cup  I,
\end{equation*}
where  $\mu$ is a probability  measure on $[-1,1]$.  For the first derivative  we have
\begin{equation*}
\varphi' (z )      =   \varphi' (0) \int_{ [-1,1] } \frac {1}{ (1-tz)^2 } d\mu(t),\quad z\in  \mathbf{C} \backslash\mathbf{R}\cup  I.
\end{equation*}
Let us note a simple  fact:   If  $\mu-$measure  of the interval  $(-1,0)$ is equal to $0$,  then  $\varphi' (t)$  is  increasing on  $[0,1)$.

Using the mentioned    integral representation of  a  function  in the Pick class   $\mathrm {P}(-1,1)$, and the L\"{o}wner's theorem
\cite[Theorem 2.7.7]{HIAI.INTER} which  says that if  $\varphi\in \mathrm{P}(I)$, then  $\varphi|_I\in  \mathrm{OM}(I)$,    and every $\varphi\in\mathrm
{OM}(I)$ is a restriction of a function in $\mathrm {P}(I)$, Joci\'{c} \cite[Lemma 2.1]{JOCIC.JFA}   proved the  following   lemma   for  $I = (-1,1)$.

\begin{lemma}\label{LE.JOCIC}
An operator monotone function $\varphi$ on an open  interval $I$   satisfies
\begin{equation*}
\varphi(t) - \varphi (s)\le \sqrt {  \varphi'(t )     \varphi'(s)} (t-s), \quad t\in I,\, s\in I,\, s<t.
\end{equation*}
\end{lemma}

A  proof of  Lemma \ref{LE.JOCIC}  is similar as in the particular case $I = (-1,1)$. However, it   may be deduced from this  special case.  Indeed, it is
enough to note  that   if $\varphi$ belongs to the class  $\mathrm {OM} (I)$, where  $I$ is   a  finite interval, then  $\varphi \circ g$ belongs to
the class  $\mathrm {OM} (-1,1)$, where $g$ is an increasing linear mapping  such that  $g(-1,1) = I$. If $I$ is   infinite interval, and if $\varphi\in
 \mathrm {OM} (I)$ then  we have $\varphi\in \mathrm {OM}(I')$  for  any finite subinterval  $I'\subseteq I$.

\begin{theorem}\label{TH.JOCIC.NORMED}
Let $\Omega\subseteq X$ be a convex domain in a normed space $X$. Assume that  $\varphi$ is a    non-constant operator  monotone function on $(-\delta,M)$,
$\delta>0$,  $M =  \sup_{x\in \Omega }\|x\|$, and let $\varphi'$ be increasing on $[0,M)$. For a Fr\'{e}chet differentiable mapping $f:\Omega\rightarrow Y$,
the Bloch type semi-norm
\begin{equation*}
\|f\|_{\mathfrak{B}, \varphi'}= \sup_{x\in  \Omega }   \frac{\|d_f(x)\|}{\varphi'(\|x\|)}
\end{equation*}
may be expressed as
\begin{equation*}
\|f\|_{\mathfrak{B},\varphi'}=  \sup_{(x,y)\in \Delta^c_\Omega}  \frac {\|f(x) - f(y )\|}{\sqrt {\varphi'(\|x\|)} \|x-y\|  \sqrt{\varphi'(\|y\|)} }.
\end{equation*}
\end{theorem}

\begin{proof}
By  Theorem \ref{TH.PAVLOVIC.MS}  it is enough to prove that
\begin{equation*}
\Psi(x,y) = {{\varphi'(\|x\|)^ {- \frac12 } } {\varphi'(\|y\|)^{-\frac 12} }},\quad x,y\in \Omega,
\end{equation*}
is an admissible function for the mapping $f$ with respect to   the     weight functions   $\omega = {\varphi'} $ on $\Omega$ and
$\tilde {\omega}  \equiv 1$ on   $Y$. Obviously,  $\Psi$  is continuous and  symmetric  on  $\Omega\times \Omega$, and $\Psi(x,x) =
\frac{1} { \varphi'(\|x\|)}$ for   $x\in  \Omega$.   It remains   to prove the inequality
\begin{equation*}
\Psi (x,y) d_{\varphi'} (x,y)\le    \|x-y\|, \quad x,\, y\in \Omega.
\end{equation*}
Indeed,     if $\gamma\subseteq\Omega$ is among  paths that connect  $x$ and $y$, and if    (for example)   $\|x\|<\|y\|$ (if $\|x\|=
\|y\|$ the proof below  is trivial),  we have
\begin{equation*}\begin{split}
d_{\varphi'} (x,y) &=  \inf_{\gamma} \int_\gamma   {\varphi'}
\le  \int _{[x,y]} {\varphi'}
 = \|x-y\| \int _0 ^1 \varphi' (\| (1-t) x + t y \|) dt
\\& \le  \|x-y\| \int _0 ^1 \varphi' ( (1-t)  \|x\| +  t \|y \|) dt
\\&= \frac{ \|x-y\|}{\|y\| - \|x\|}   \int _0 ^1 \frac {d}{dt}   (\varphi  ( \|x\| +  t (\|y \| -\|x\|)) )
\\&  = \|x-y\| \frac{ \varphi(\|y\|) - \varphi (\|x\|)}{\|y\| - \|x\|}
  \le \|x-y\|  \sqrt{\varphi (\|x\|)}   \sqrt{\varphi (\|y\|)}  \\& =\|x-y\|\Psi (x,y)^{-1},
\end{split}\end{equation*}
where we have used Lemma \ref{LE.JOCIC} in the last inequality.
\end{proof}

The     following corollary  is  generalization of Proposition  \ref{PR.PAVLOVIC} for normed  spaces  and  Fr\'{e}chet differentiable
mappings. Bloch spaces on domains in a Banach space  have been recently  considered  in \cite{BLASCO.JFA, CHUA.JFA}.

\begin{corollary}\label{COR.PAVLOVIC}
Let $X$ and $Y$ be normed spaces. For a Fr\'{e}chet differentiable mapping $f:{B}_X \to   Y$ the  Bloch semi-norm
\begin{equation*}
\|f\| _{\mathfrak{B} }  = \sup_{x\in{B}_X} (1-\|x\|^2)   \| d_f (x)\|
\end{equation*}
is equal to
\begin{equation*}
\|f\| _{\mathfrak{B}}
=\sup_{(x,y)\in \Delta^c_{B_X}} (1-\|x\|^2)^{\frac 12} (1-\|y\|^2)^{\frac 12}  \frac{\|f(x)-f (y)\|}{\|x-y\|}.
\end{equation*}
\end{corollary}

This corollary         follows  from the  Theorem \ref{TH.JOCIC.NORMED} if we take  $\varphi  = \tanh^{-1} \in  \mathrm { OM}(-1,1)$.

\subsection{The case of weights monotone in norm}
In Theorem \ref{TH.MIN.MAX},  which particular case (for $\tilde{\Omega}=Y$  and  $\tilde{\omega} \equiv 1$) is  mentioned in  the Introduction, we have
found   a very simple admissible function for weight functions  monotone  in norm  on domains of normed spaces. It would be of some  interest to find
a  such simple admissible function for other domains and/or for other     weight  functions not  necessary monotone  in  norm.

\begin{theorem}\label{TH.MIN.MAX}
Let   $\omega(x)$ be a weight  function on  a convex  domain $\Omega$ in  a normed space $X$ which is  monotone in  $\|x\|$, and let $\tilde \omega (y)$
be a weight  function   on    a  domain $\tilde{\Omega}$  in a normed space $Y$  which is   monotone    in   $\|y\|$. The  Bloch number
\begin{equation*}
\mathfrak{B}_f  = \sup_{x \in \Omega}\frac { \tilde{\omega} (f(x))}{\omega(x) } \|d_f(x)\|
\end{equation*}
of a Fr\'{e}chet  differentiable mapping $f:\Omega\rightarrow \tilde{\Omega}$ is equal  to
\begin{equation*}
\mathfrak{B} _f = \sup_{(x,y)\in\Delta^c_\Omega}   \frac{\min\{\tilde{\omega} (f(x)),  \tilde{\omega} (f(y))\}}{ \max \{\omega (x),\omega(y)\}}
\frac {\|f(x) - f(y)\|}{\|x - y\|}.
\end{equation*}
\end{theorem}

\begin{proof}
In order  to apply  Theorem \ref{TH.PAVLOVIC.MS} for  $\Omega\subseteq X$  and  $\tilde{\Omega}\subseteq Y$  with the metrics induces by the norms on $X$
and $Y$, respectively, we have  to  prove  that
\begin{equation*}
\Psi _f (x,y) =   \frac{\min \{\tilde{\omega}(f(x)),\tilde{\omega}(f(y))\}}{\max \{\omega(x), \omega(y)\}}, \quad x,\, y\in\Omega
\end{equation*}
is an admissible function  for the  mapping  $f:\Omega\to \tilde{\Omega}$  with respect to $\omega$ and $\tilde{\omega}$. Since it is obviously symmetric
and  continuous on $\Omega\times \Omega$, and $\Psi _f (x,x)  = \frac{ \tilde{\omega}(f(x))} { \omega(x)}$ for $x\in \Omega$,  we have only to  show that
$\Psi _f$ satisfies
\begin{equation*}\begin{split}
\Psi_f(x,y)\frac{\|f(x)-f(y)\|}{\|x-y\|}\le \frac {d_{\tilde {\omega}}  (f(x),f(y))}{d_\omega (x,y)},\quad x,\, y\in \Omega.
\end{split} \end{equation*}

Since       $\omega(z)$ is monotone    in $\|z\|$, and since $\Omega$ is convex, we have
\begin{equation*}\begin{split}
d_\omega (x,y)&   \le\int_{[x,y]} {\omega}  \le \max_{z\in [x,y]}{\omega(z)}\|x-y\| = \max \{\omega(x),\omega(y)\} \|x-y\|
\end{split} \end{equation*}
for $x\in\Omega $  and $y\in\Omega $.

If $\tilde{\omega}(z)$ is increasing in $\|z\|$, let $\tilde{\gamma}:[a,b]\to Y$      be among paths in $\Omega$  that connect $x\in \Omega $ and $y\in
\Omega$ and satisfy  $\| \tilde{\gamma} (t) \| \le  \max\{ \|x\|,\|y\|\}$, $t\in  [a,b]$. On the other hand, if   $\omega (z)$ is  a decreasing function
of $\|z\|$, then we   should   consider  only  the paths $\tilde{\gamma}$ such that $\| \tilde{\gamma} (t) \| \ge  \min\{ \|x\|,\|y\|\}$).
Let $\gamma\subseteq  \Omega $ be among paths that connect  $x$ and $y$  without any restriction.   Now,   we have
\begin{equation*}\begin{split}
d_{\tilde{\omega}} (x,y)& = \inf_\gamma  \int_\gamma \tilde{\omega}
= \inf_{\tilde {\gamma}}  \int_{\tilde {\gamma}}
 \tilde{\omega}  \ge \inf_{\tilde {\gamma}}  \{ \min_{z\in\tilde{\gamma}}{\tilde{\omega}(z)}\|x- y\|\}\ge
\min \{\tilde{\omega}(x),\tilde{\omega}(y)\} \|x-y\|.
\end{split} \end{equation*}

It follows
\begin{equation*}\begin{split}
\frac {d_{\tilde {\omega}}  (f(x),f(y))}{d_\omega (x,y)}
& \ge \frac{\min \{\tilde{\omega}(f(x)),\tilde{\omega}(f(y))\} }{\max \{\omega(x),\omega(y)\}}
 \frac{\|f(x)-f(y)\|}{\|x-y\|},
\end{split} \end{equation*}
which proves the fourth condition  of   admissibility.
\end{proof}

\subsection{Characterizations of Bloch and  normal mappings}
Recall  that  an analytic function  $f$  on the unit disc    is  a  Bloch function if it satisfies the  growth   condition
\begin{equation*}
|f'(z)|\le\frac C{1-|z|^2},\quad  z\in \mathbf{U},
\end{equation*}
where $C$ is a constant. For equivalent definitions we refer to \cite{ANDERSON.JRAM}. On the  other  hand, an analytic function $f$ is normal on the
unit disc if   the family
\begin{equation*}
\{f\circ\varphi: \varphi\, \text{is a  Moebius transform of}\, \mathbf{U}\}
\end{equation*}
is a normal family \cite{LEHTO.ACTA}. However, it is known that this condition is    equivalent to the growth  condition
\begin{equation*}
\frac {|f'(z)|}{1+|f(z)|^2}\le    \frac C{1-|z|^2},\quad   z\in \mathbf{U},
\end{equation*}
where  $C$ is a constant. For similarities between   analytic  Bloch and  normal  functions  we refer   to  \cite{ANDERSON.JRAM, COLONNA.PALERMO}.

Let $\rho$  be the hyperbolic distance on the unit  disc   $\mathbf{U}$, and let $\sigma$ be the spherical  distance on $\mathbf{C}$. It is very well
known that an analytic  function is   Bloch  function   if and only if it is  Lipschitz continuous  with respect to  the Euclidean   distance and the  hyperbolic distance, respectively, i.e.,
\begin{equation*}
| f(z) - f(w)| \le C \rho(z,w),\quad   z,\, w\in \mathbf{U},
\end{equation*}
as well as  that the   function $f$ is  normal  if and only if  it  satisfies
\begin{equation*}
\sigma( f(z),f(w)) \le  C\rho(z,w),\quad  z,\, w\in \mathbf{U}.
\end{equation*}
It  is also well  known that these conditions are equivalent to the corresponding  growth conditions.      This follows  also from the remark  after
Corollary  \ref{COR.M}.

The aim of this last subsection  is to  give a   more direct proof of Proposition \ref{PR.PAVLOVIC} following and simplifying the proof form
\cite{MARKOVIC.CMB}, and to consider the Bloch type growth  condition  which satisfies a normal  function. We derive a new criteria for  normality of
analytic mapping    on the unit disc.

In    Theorem \ref{TH.PAVLOVIC.MS} let us    take $X = {B}^m$ with the standard distance, and  let $Y$  be  a  normed space. Moreover, let  $\omega(x)
=  (1-|x|^2)^{-1}$,  $x\in {B}^m$  and $\tilde {\omega}\equiv 1$  be the weigh functions. Then $d_{\omega}$ is the hyperbolic  distance on the unit
ball  ${B}^m$, which  in the sequel we  denote  by $\rho $,  and $d_{\tilde{\omega}}$ is  the distance on $Y$  produced by  the     norm $\|\cdot\|_Y$.
The  hyperbolic distance  on  ${B}^m$ is  given      explicitly      by
\begin{equation*}
\rho (x,y)  =  \mathrm {asinh} \frac{|x - y |}{\sqrt {1-|x|^2}\sqrt {1-|y^2}}, \quad x,\, y\in {B}^m
\end{equation*}
(see \cite{VUORINEN.BOOK}). We will prove  that an  admissible function in this  setting  is
\begin{equation*}
\Psi (x, y)   = \sqrt{1-|x|^2} \sqrt{1-|y|^2},\quad   x,\,  y\in B^m.
\end{equation*}
We only   have to show  that  $\rho (x,y)  \Psi  (x,y)\le |x-y|$, $x,\, y\in B^m$. By   the  inequality $\mathrm {asinh}\, t\le t$, $t\ge0$ (indeed,
let $\phi(t) =  \mathrm {asinh}(t)  -t$; then we have $\phi(0)=0$, and $\phi'(t)  =    \frac{1}{\sqrt{1+t^2} }-1<0$ for   $t>0$, so our inequality
follows  since  $\phi(t)\le \phi (0) = 0$), we deduce
\begin{equation*}\begin{split}
\frac { |x - y|}{\rho (x,y) } &
= {|x-y |} : { \mathrm {asinh}  \frac{|x - y  |}{\sqrt {1-|x|^2}\sqrt {1-|y|^2}}} \ge   {|x-y |} : {\frac{|x - y  |}{\sqrt {1-|x|^2}\sqrt {1-|y|^2}}}
\\&=  \sqrt {1-|x|^2} \sqrt {1-|y|^2}   =    \Psi(x,y),\quad x,\,  y\in {B}^m.
\end{split}\end{equation*}
Proposition \ref{PR.PAVLOVIC}  now follows form Theorem \ref{TH.PAVLOVIC.MS}.

As we have  said the   growth condition given in the corollary below   is satisfied   by  the class of normal functions on the unit disc.

\begin{corollary}
A differentiable  mapping $f: {B} ^m\rightarrow \mathbf{R}^2$     satisfies the   growth  condition
\begin{equation*}
\frac{\|d_f(x)\|} { 1-|x|^2}\le   \frac {C}{ 1+|f(x)|^2}, \quad x\in {B}^m,
\end{equation*}
where  $C$ is a constant,  if and only if
\begin{equation*}
|f(x) -  f(y)| \le C  |x-y|\frac{\sqrt{1+|f(x)| ^2} \sqrt{1+|f(y)|^2}}{\sqrt{1-|x|^2} \sqrt{1-|y|^2}},\quad x,\,  y\in {B}^m.
\end{equation*}
\end{corollary}

\begin{proof}
In Theorem \ref{TH.PAVLOVIC.MS} let us take  $X = {B} ^m$ and  $Y= \mathbf{R}^2$    with Euclidean distances, and let
\begin{equation*}
\omega (x) = ({1-|x|^2})^{-1},\quad  x\in {B}^m,\quad  \tilde{\omega}(z)= ({1+|z|^2})^{-1},\quad  z\in \mathbf{R}^2
\end{equation*}
be the weight functions on these domains. As we have  already said, $d_{\omega}$ is  the hyperbolic  distance $\rho$ on  the unit ball $ {B}^m$.
On the other hand,  $d_{\tilde{\omega}}  $ is  the  spherical distance $\sigma $  on $\mathbf{R}^2$.
It is well known that
\begin{equation*}
\sigma (z,w ) = \frac{|z-w|}{\sqrt{1+|z|^2} \sqrt{1+|w|^2}},\quad z,\,  w\in \mathbf{R}^2.
\end{equation*}

We are going  to show  that
\begin{equation*}
\Psi _f (x,y) =\frac{\sqrt{1-|x|^2} \sqrt{1-|y|^2}}{\sqrt{1+|f(x)|^2} \sqrt{1+|f(y)|^2}},\quad x,\, y\in B^m
\end{equation*}
is an  admissible  function for the mapping  $f$ with respect to the hyperbolic and  spherical weights. Since
\begin{equation*}
\Psi_f (x,x)  = \frac{ 1-|x|^2 }{ 1+|f(x)|^2 }  =  \frac{ 1 }{ 1+|f(x)|^2 } : \frac{ 1 }{ 1-|x|^2 },\quad  x\in  {B}^m,
\end{equation*}
and  since  $\Psi_f(x,y)$  is obviously  symmetric and continuous, it remains only  to prove that  $\Psi_f(x,y)$ satisfies
\begin{equation*}
\Psi_f(x,y)  \frac{|f(x)-f(y)|}{|x-y|} \le \frac {\sigma (f(x),f(y))}{\rho (x,y)}, \quad x,\,   y\in  {B} ^m.
\end{equation*}
Having on mind  the  inequality $\mathrm {asinh}\, t\le t$ for $t\ge  0$, we obtain
\begin{equation*}\begin{split}
\frac{\sigma (f(x),f(y) )}{\rho(x,y)}
& =  { \frac{|f(x)-f(y)|}{\sqrt{1+|f(x)|^2} \sqrt{1+|f(y)|^2}}}
: {\mathrm{asinh}  \frac{|x-y|}{\sqrt{1-|x|^2} \sqrt{1-|y|^2}}}
\\&\ge { \frac{|f(x)-f(y)|}{\sqrt{1+|f(x)|^2} \sqrt{1+|f(y)|^2}}}
 : {\frac{|x-y|}{\sqrt{1-|x|^2} \sqrt{1-|y|^2}}}
\\&= \frac{\sqrt{1-|x|^2} \sqrt{1-|y|^2}}{\sqrt{1+|f(x)|^2} \sqrt{1+|f(y)|^2}}
 \frac{|f(x)-f(y)|}{|x-y|} \\&= \Psi_f(x,y)  \frac{|f(x)-f(y)|}{|x-y|},
\end{split}\end{equation*}
which is the inequality we aimed to prove.
\end{proof}

Based   on the preceding results,  we   now  state the characterisations of  Bloch  and  normal functions  on  the  unit disc.

\begin{proposition}
Let $f:\mathbf{U}  \to \mathbf{C}$  be an analytic function.

\begin{equation*}
f \ \text{is a Bloch if and only if}\ {\sqrt{1-|z|^2} \sqrt{1-|w|^2}} \frac {|f(z) -  f(w)|}{|z-w|}\ \text{is bounded  for}\   z\ne w.
\end{equation*}

\begin{equation*}
f \ \text{is a normal iff}\  \frac{\sqrt{1-|z|^2} \sqrt{1-|w|^2}}{\sqrt{1+|f(z)|^2} \sqrt{1+|f(w)|^2}} \frac {|f(z) -  f(w)|}{|z-w|}\ \text{is bounded  for}\    z\ne w.
\end{equation*}
\end{proposition}

\end{document}